\documentclass[11pt,leqno]{amsart}

\usepackage{amsmath,amsthm}
\usepackage{amssymb}
\usepackage{color}
\newtheorem{theorem}{THEOREM}
\newtheorem{Prop}{PROPOSITION}
\newtheorem{Cor}{COROLLARY}
\newtheorem{Lem}{LEMMA}

\allowdisplaybreaks[4]

\begin{document} 

%\title{Remarks on hypergeometric Cauchy numbers}  
\title{\bf REMARKS ON HYPERGEOMETRIC CAUCHY NUMBERS}

\author{MIHO AOKI}
\address{ 
Department of Mathematics\\
Shimane University\\ 
Matsue, Shimane, Japan\\ 
}
\email{aoki@riko.shimane-u.ac.jp}
\thanks{The first author was supported by JSPS KAKENHI Grant Numbers 26400015.}
\author{TAKAO KOMATSU}
\address{ 
School of Mathematics and Statistics\\
Wuhan University\\
Wuhan 430072 China\\
}
\email{komatsu@whu.edu.cn}
%

%\author{
%MIHO AOKI\thanks{
%The first author was supported by JSPS KAKENHI Grant Numbers 26400015.
%}\\
%\small Department of Mathematics\\
%\small Shimane University\\ 
%\small Matsue, Shimane, Japan\\ 
%\small \texttt{aoki@riko.shimane-u.ac.jp}\\\\  
%TAKAO KOMATSU \\ 
%\small School of Mathematics and Statistics\\
%\small Wuhan University\\
%\small Wuhan 430072 China\\
%\small \texttt{komatsu@whu.edu.cn}
%}

\maketitle 
{\tiny
\begin{abstract}
For a positive integer $N$, hypergeometric Cauchy numbers $c_{N,n}$ are defined by 
$$
\frac{1}{{}_2 F_1(1,N;N+1;-x)}=\frac{(-1)^{N-1}x^N/N}{\log(1+x)-\sum_{n=1}^{N-1}(-1)^{n-1}x^n/n}=\sum_{n=0}^\infty c_{N,n}\frac{x^n}{n!}\,, 
$$  
where ${}_2 F_1(a,b;c;z)$ is the Gauss hypergeometric function. 
When $N=1$, $c_n=c_{1,n}$ are the classical Cauchy numbers. 
In 1875, Glaisher gave several interesting determinant expressions of numbers, including Bernoulli, Cauchy and Euler numbers (see (\ref{det:ber}), (\ref{det:cau}) and (\ref{det:eul}) in the text). 
%%%%%%%%%%%%%%%%%%
%%%%%%%%%%%%%%%%%
%%%%%%%%%%%%%%%%%%%
%\begin{color}{red}  
Hypergeometric numbers can be recognized as one of the most natural extensions of the classical Cauchy numbers in terms of determinants (see Section \ref{sec:det}), though  many kinds of generalizations of the Cauchy numbers have been considered by many authors.   
%\end{color}{red} 
%One of the advantages of hypergeometric Cauchy numbers is the natural extension of determinant expressions of the numbers (see Section 2), though  many kinds of generalizations of the Cauchy numbers have been considered by many authors. 
%%%%%%%%%%%%%%%%%%
%%%%%%%%%%%%%%%%%
%%%%%%%%%%%%%%%%%%%
In addition, there are some relations between the hypergeometric Cauchy numbers and the classical Cauchy numbers.  
In this paper, 
we give the determinant expressions of hypergeometric Cauchy numbers and their generalizations, and show some 
interesting expressions of hypergeometric Cauchy numbers.
%%%%%%%%%%%%%%%%%%
%\begin{color}{red} 
As applications, we can get the inversion relations such that hypergeometric Cauchy numbers as $c_{N,n}/n!$ and the numbers $N/(N+n)$ are interchanged in terms of determinants of the so-called Hassenberg matrices.  
%\end{color}{red} 
%Furthermore, hypergeometric Cauchy numbers as $c_{N,n}/n!$ have the inversion relation with the numbers $N/(N+n)$ in terms of determinants of the so-called Hassenberg matrices. 
%%%%%%%%%%%%%%%%%%
%%%%%%%%%%%%%%%%%
%%%%%%%%%%%%%%%%%%%
\end{abstract}
}
\\ \\
{\bf Abbreviated title:}   \\REMARKS ON HYPERGEOMETRIC CAUCHY NUMBERS

\begin{center}
\section{INTRODUCTION} 
\end{center}

Denote ${}_2 F_1(a,b;c;z)$ be the Gauss hypergeometric function defined by 
$$
{}_2 F_1(a,b;c;z)=\sum_{n=0}^\infty\frac{(a)^{(n)}(b)^{(n)}}{(c)^{(n)}}\frac{z^n}{n!}  
$$ 
with the rising factorial $(x)^{(n)}=x(x+1)\dots(x+n-1)$ ($n\ge 1$) and $(x)^{(0)}=1$. 
For $N\ge 1$, define the hypergeometric Cauchy numbers $c_{N,n}$ (\cite{Ko3}) by 
\begin{equation}  
\frac{1}{{}_2 F_1(1,N;N+1;-x)}=\frac{(-1)^{N-1}x^N/N}{\log(1+x)-\sum_{n=1}^{N-1}(-1)^{n-1}x^n/n}=\sum_{n=0}^\infty c_{N,n}\frac{x^n}{n!}\,.  
\label{def:hgc}
\end{equation}     
When $N=1$, $c_n=c_{1,n}$ are classical Cauchy numbers (\cite{C}) defined by 
$$
\frac{x}{\log(1+x)}=\sum_{n=0}^\infty c_n\frac{x^n}{n!}\,. 
$$ 
Notice that $b_n=c_n/n!$ are sometimes called the Bernoulli numbers of the second kind.  
In addition, the hypergeometric Cauchy polynomials $c_{M,N,n}(z)$ (\cite{Ko3}) are defined by 
\begin{align*} 
\frac{1}{(1+x)^z}\frac{1}{{}_2F_1(M,N;N+1;-x)}&=\frac{1}{(1+x)^z}\sum_{n=0}^\infty c_{M,N,n}\frac{x^n}{n!}\\
&=\sum_{n=0}^\infty c_{M,N,n}(z)\frac{x^n}{n!}\,. 
\end{align*}  
so that $c_{1,N,n}(0)=c_{N,n}$.  

Similar hypergeometric numbers are hypergeometric Bernoulli numbers $B_{N,n}$ and hypergeometric Euler numbers.  
For $N\ge 1$, define hypergeometric Bernoulli numbers $B_{N,n}$ (\cite{HN1,HN2,Ho1,Ho2,Kamano2,Ng}) by 
\begin{equation}  
\frac{1}{{}_1 F_1(1;N+1;x)}=\frac{x^N/N!}{e^x-\sum_{n=0}^{N-1}x^n/n!}=\sum_{n=0}^\infty B_{N,n}\frac{x^n}{n!}\,,   
\label{def:hgb} 
\end{equation}  
where ${}_1 F_1(a;b;z)$ is the confluent hypergeometric function defined by 
$$
{}_1 F_1(a;b;z)=\sum_{n=0}^\infty\frac{(a)^{(n)}}{(b)^{(n)}}\frac{z^n}{n!}\,.   
$$ 
When $N=1$, $B_{1,n}=B_n$ are classical Bernoulli numbers, defined by 
$$
\frac{x}{e^x-1}=\sum_{n=0}^\infty B_n\frac{x^n}{n!}\,. 
$$  
The hypergeometric Euler numbers $E_{N,n}$ (\cite{KZ}) are defined by 
$$ 
\frac{1}{{}_1 F_2(1;N+1,(2 N+1)/2;x^2/4)}=\sum_{n=0}^\infty E_{N,n}\frac{x^n}{n!}\,, 
$$   
where ${}_1 F_2(a;b,c;z)$ is the hypergeometric function defined by 
$$
{}_1 F_2(a;b,c;z)=\sum_{n=0}^\infty\frac{(a)^{(n)}}{(b)^{(n)}(c)^{(n)}}\frac{z^n}{n!}\,.
$$
When $N=0$, then $E_n=E_{0,n}$ are classical Euler numbers defined by 
$$ 
\frac{1}{\cosh x}=\sum_{n=0}^\infty E_n\frac{x^n}{n!}\,. 
$$ 

Several kinds of generalizations of the Cauchy numbers (or the Bernoulli numbers of the second kind) have been considered by many authors. For example, poly-Cauchy numbers \cite{Ko1}, multiple Cauchy numbers, shifted Cauchy numbers \cite{KS1}, generalized Cauchy numbers \cite{KLL}, incomplete Cauchy numbers \cite{Ko5,Ko9,KMS}, various types of $q$-Cauchy numbers \cite{CK,Ko2,Ko8,KS2}, Cauchy Carlitz numbers \cite{KK1,KK2}. The situations are similar and even more for Bernoulli numbers and Euler numbers.  
%%%%%%%%%%%%%%%%%%%
%%%%%%%%%%%%%%%%%%%%
%%%%%%%%%%%%%%%%%%%%%%
%\begin{color}{red}  
Hypergeometric numbers can be recognized as one of the most natural extensions of the classical numbers in terms of determinants, though  many kinds of generalizations of the classical numbers have been considered by many authors.   
%\end{color}{red} 
%One of the advantages of hypergeometric numbers is the natural extension of determinant expressions of the numbers. 
%%%%%%%%%%%%%%%%%%%%%
%%%%%%%%%%%%%%%%%%%%%
%%%%%%%%%%%%%%%%%%%%%
In \cite{Ko10,KZ}, the hypergeometric Euler numbers $E_{N,2n}$ can be expressed as 
$$
E_{N, 2 n}=(-1)^n(2 n)!   
\left|
\begin{array}{cccc}
\frac{(2 N)!}{(2 N+2)!}&1&&\\
\frac{(2 N)!}{(2 N+4)!}&\ddots&\ddots&\\
\vdots&&\ddots&1\\
\frac{(2 N)!}{(2 N+2 n)!}&\cdots&\frac{(2 N)!}{(2 N+4)!}&\frac{(2 N)!}{(2 N+2)!}
\end{array} 
\right|\quad(N\ge 0,~n\ge 1)\,. 
$$ 
When $N=0$, this is reduced to a famous determinant expression of Euler numbers ({\it cf.} \cite[p.52]{Glaisher}): 
\begin{equation} 
E_{2n}=(-1)^n (2n)! 
\begin{vmatrix}   
\frac{1}{2!}& 1 &~& ~&~\\
\frac{1}{4!}&  \frac{1}{2!} & 1 &~&~\\
\vdots & ~  &  \ddots~~ &\ddots~~ & ~\\
\frac{1}{(2n-2)!}& \frac{1}{(2n-4)!}& ~&\frac{1}{2!} &  1\\
\frac{1}{(2n)!}&\frac{1}{(2n-2)!}& \cdots &  \frac{1}{4!} & \frac{1}{2!}
\end{vmatrix}\,.
\label{det:eul}
\end{equation}  
In addition,  When $N=1$, $E_{1,n}$ can be expressed by Bernoulli numbers as 
$E_{1,n}=-(n-1)B_n$ (\cite{KZ}). 
 
In \cite{AK}, the hypergeometric Bernoulli numbers $B_{N,n}$ can be expressed as 
$$
B_{N,n}=(-1)^n n!\left|
\begin{array}{ccccc} 
\frac{N!}{(N+1)!}&1&&&\\  
\frac{N!}{(N+2)!}&\frac{N!}{(N+1)!}&&&\\ 
\vdots&\vdots&\ddots&1&\\ 
\frac{N!}{(N+n-1)!}&\frac{N!}{(N+n-2)!}&\cdots&\frac{N!}{(N+1)!}&1\\ 
\frac{N!}{(N+n)!}&\frac{N!}{(N+n-1)!}&\cdots&\frac{N!}{(N+2)!}&\frac{N!}{(N+1)!}
\end{array} 
\right|\quad(N\ge 1,~n\ge 1)\,.  
$$ 
When $N=1$, we have a determinant expression of Bernoulli numbers (\cite[p.53]{Glaisher}):    
\begin{equation}  
B_n=(-1)^n n!\left|
\begin{array}{ccccc} 
\frac{1}{2!}&1&&&\\  
\frac{1}{3!}&\frac{1}{2!}&&&\\ 
\vdots&\vdots&\ddots&1&\\ 
\frac{1}{n!}&\frac{1}{(n-1)!}&\cdots&\frac{1}{2!}&1\\ 
\frac{1}{(n+1)!}&\frac{1}{n!}&\cdots&\frac{1}{3!}&\frac{1}{2!}
\end{array} 
\right|\,. 
\label{det:ber} 
\end{equation}
In addition, relations between $B_{N,n}$ and $B_{N-1,n}$ are shown in \cite{AK}.

In this paper, we shall give  similar determinant expression of hypergeometric Cauchy numbers and their generalizations.  
We also study some interesting relations between the hypergeometric Cauchy numbers and the classical Cauchy numbers. 
%%%%%%%%%%%%%%%%%%%%%%%%
%%%%%%%%%%%%%%%%%%%%%%%%
%%%%%%%%%%%%%%%%%%%%%%%%%%
%\begin{color}{red} 
As applications, we can get the inversion relations such that hypergeometric Cauchy numbers as $c_{N,n}/n!$ and the numbers $N/(N+n)$ are interchanged in terms of determinants of the so-called Hassenberg matrices.  
%In fact, hypergeometric Cauchy numbers as $c_{N,n}/n!$ have the inversion relation with the numbers $N/(N+n)$ in terms of determinants of the so-called (lower) Hassenberg matrices.  
%\end{color}
%%%%%%%%%%%%%%%%%%%%%%%%%%
%%%%%%%%%%%%%%%%%%%%%%%%%%%
%%%%%%%%%%%%%%%%%%%%%%%%
%\section{Some basic properties of hypergeometric Cauchy numbers}   
%
%
%
\begin{center}
\end{center}
\begin{center}
\section{A DETERMINANT EXPRESSION OF THE HYPERGEOMETRIC CAUCHY NUMBERS}
\label{sec:det} 
\end{center}

From the definition (\ref{def:hgc}), we have 
\begin{equation}  
\sum_{i=0}^n\frac{(-1)^i c_{N,i}}{(N+n-i)i!}=0\quad(n\ge 1)
\label{rel:hgc} 
\end{equation}  
with $c_{N,0}=1$ (\cite[Proposition 1]{Ko3}).

By using this expression, the first few values of $c_{N,n}$ are given:  
{\small
\begin{align*}  
c_{N,0}&=1\,,\\
c_{N,1}&=\frac{N}{N+1}\,,\\
c_{N,2}&=-\frac{2 N}{(N+1)^2(N+2)}\,,\\
c_{N,3}&=\frac{6 N(N^2+N+2)}{(N+1)^3(N+2)(N+3)}\,,\\
c_{N,4}&=-\frac{4!N(N^5+5 N^4+14 N^3+24 N^2+20 N+12)}{(N+1)^4(N+2)^2(N+3)(N+4)}\,,\\
c_{N,5}&=\frac{5!N(N^7+8 N^6+35 N^5+96 N^4+160 N^3+184 N^2+116 N+48)}{(N+1)^5(N+2)^2(N+3)(N+4)(N+5)}\,. 
\end{align*}  
}
In \cite[Theorem 1]{Ko3}, an explicit expression of hypergeometric Cauchy numbers is given.
\\
\begin{Lem}  
For $N,n\ge 1$, we have 
$$ 
c_{N,n}=(-1)^n n!\sum_{r=1}^n(-N)^r\sum_{i_1+\cdots+i_r=n\atop i_1,\dots,i_r\ge 1}\frac{1}{(N+i_1)\cdots(N+i_r)}\,.
$$ 
\label{lem-hgc}
\end{Lem}  

Such values of $c_{N,n}$ can be expressed in terms of the determinant.

\begin{theorem}  
For $N,n\ge 1$, we have 
$$
c_{N,n}=n!\left|
\begin{array}{ccccc} 
\frac{N}{N+1}&1&&&\\  
\frac{N}{N+2}&\frac{N}{N+1}&&&\\ 
\vdots&\vdots&\ddots&1&\\ 
\frac{N}{N+n-1}&\frac{N}{N+n-2}&\cdots&\frac{N}{N+1}&1\\ 
\frac{N}{N+n}&\frac{N}{N+n-1}&\cdots&\frac{N}{N+2}&\frac{N}{N+1}
\end{array} 
\right|\,. 
$$ 
\label{th-cau1} 
\end{theorem} 

\noindent 
{\it Remark.}   
When $N=1$, we have a determinant expression of Cauchy numbers (\cite[p.50]{Glaisher}):  
\begin{equation} 
c_n=n!\left|
\begin{array}{ccccc} 
\frac{1}{2}&1&&&\\  
\frac{1}{3}&\frac{1}{2}&&&\\ 
\vdots&\vdots&\ddots&1&\\ 
\frac{1}{n}&\frac{1}{n-1}&\cdots&\frac{1}{2}&1\\ 
\frac{1}{n+1}&\frac{1}{n}&\cdots&\frac{1}{3}&\frac{1}{2}
\end{array} 
\right|\,. 
\label{det:cau}
\end{equation}
The value of this determinant, that is, $b_n=c_n/n!$ are called {\it Bernoulli numbers of the second kind}.  

\begin{proof}[Proof of Theorem \ref{th-cau1}] 
Put $b_{N,n}=c_{N,n}/n!$.  
Then, we shall prove that 
\begin{equation}  
b_{N,n}=\left|
\begin{array}{ccccc} 
\frac{N}{N+1}&1&&&\\  
\frac{N}{N+2}&\frac{N}{N+1}&&&\\ 
\vdots&\vdots&\ddots&1&\\ 
\frac{N}{N+n-1}&\frac{N}{N+n-2}&\cdots&\frac{N}{N+1}&1\\ 
\frac{N}{N+n}&\frac{N}{N+n-1}&\cdots&\frac{N}{N+2}&\frac{N}{N+1}
\end{array} 
\right|\,. 
\label{express-b2}
\end{equation} 
Since by (\ref{rel:hgc}) 
$$
c_{N,n}=\sum_{i=0}^{n-1}(-1)^{n-i-1}\frac{n!}{i!}\frac{N}{N+n-i}c_{N,i}\,,
$$  
we have 
$$
b_{N,n}=N\sum_{i=0}^{n-1}\frac{(-1)^{n-i-1}}{N+n-i}b_{N,i}\,. 
$$ 
When $n=1$, (\ref{express-b2}) is true as $b_{N,1}=N/(N+1)$.  
Assume that the result (\ref{express-b2}) is true up to $n-1$.  
Then, by expanding the last column of the right-hand side of (\ref{express-b2}), it is equal to  
\begin{align*}  
&\frac{N}{N+1}b_{N,n-1}
-\left|
\begin{array}{ccccc} 
\frac{N}{N+1}&1&&&\\  
\frac{N}{N+2}&\frac{N}{N+1}&&&\\ 
\vdots&\vdots&\ddots&1&\\ 
\frac{N}{N+n-2}&\frac{N}{N+n-3}&\cdots&\frac{N}{N+1}&1\\ 
\frac{N}{N+n}&\frac{N}{N+n-1}&\cdots&\frac{N}{N+3}&\frac{N}{N+2}
\end{array} 
\right|\\
&=\frac{N}{N+1}b_{N,n-1} 
-\frac{N}{N+2}b_{N,n-2} 
+\left|
\begin{array}{ccccc} 
\frac{N}{N+1}&1&&&\\  
\frac{N}{N+2}&\frac{N}{N+1}&&&\\ 
\vdots&\vdots&\ddots&1&\\ 
\frac{N}{N+n-3}&\frac{N}{N+n-4}&\cdots&\frac{N}{N+1}&1\\ 
\frac{N}{N+n}&\frac{N}{N+n-1}&\cdots&\frac{N}{N+4}&\frac{N}{N+3}
\end{array} 
\right|\\
&=\frac{N}{N+1}b_{N,n-1}-\frac{N}{N+2}b_{N,n-2}+\frac{N}{N+3}b_{N,n-3}\\
&\quad-\cdots
+(-1)^{n-2}\left|
\begin{array}{cc}
\frac{N}{N+1}&1\\
\frac{N}{N+n}&\frac{N}{N+n-1} 
\end{array}\right|
\\
&=N\sum_{i=0}^{n-1}\frac{(-1)^{n-i-1}}{N+n-i}b_{N,i}=b_{N,n}\,. 
\end{align*} 
Note that 
$$ 
b_{N,1}=\frac{N}{N+1}\quad\hbox{and}\quad b_{N,0}=1\,. 
$$  
\end{proof} 
 \begin{center}
\section{A RELATION BETWEEN  $c_{N,n}$ AND $c_{N-1,n}$}
\end{center}

In this section, we show the following relation between $c_{N,n}$ and $c_{N-1,n}$.
\\
\begin{Prop}
For $N\geq 2$ and $n\geq 1$, we have
$$
c_{N,n}= \sum_{m=0}^{n} \left( \frac{N}{1-N} \right)^m
\sum_{0\leq i_m <\cdots <i_1 <i_0= n} \frac{n!}{i_m!} \ c_{N-1,i_m} \prod_{k=1}^m \frac{ c_{N-1, i_{k-1}-i_k+1} }{(i_{k-1}-i_k+1)!},
$$
where $i_0=n$.
\label{prp4}
\end{Prop}
\bigskip
\noindent
{\it Examples}

\begin{itemize}
\item[{\rm (i)}] $\displaystyle{ c_{N,1}=c_{N-1,1}+\frac{N}{1-N} \times \frac{ c_{N-1,0} c_{N-1,2}}{2} }$
\item[{\rm (ii)}] $\displaystyle{ c_{N,2}=c_{N-1,2}+\frac{N}{1-N} \left( \frac{c_{N-1,3}}{3}+c_{N-1,1}c_{N-1,2} \right) 
+\left(
\frac{N}{1-N} \right)^2 \times \frac{c_{N-1,2}^2}{2}}$
\end{itemize}

\begin{Lem}
For $N\geq 2$ and $n\geq 0$, we have
$$
c_{N,n}=c_{N-1,n} -\frac{N}{(n+1)(N-1)} \sum_{m=0}^{n-1}  \binom{n+1}{m} c_{N,m}c_{N-1,n-m+1}.
$$
\label{lem1}
\end{Lem}

\begin{proof}
From (\ref{def:hgc}), we have 
\begin{align*}
(-1)^{N-1} \times  \frac{x^N}{N} & =\left( \sum_{n=0}^{\infty} c_{N,n} \frac{x^n}{n!} \right) \left\{ {\rm log}(1+x)-\sum_{n=1}^{N-2} (-1)^{n-1} \frac{x^n}{n} -(-1)^{N-2} \frac{x^{N-1}}{N-1} \right\}.
\end{align*}
By dividing ${\rm log}(1+x)-\sum_{n=1}^{N-2} (-1)^{n-1}x^n/n$, we have
\begin{align*}
\frac{1-N}{N} \ x \times \sum_{n=0}^{\infty} c_{N-1,n} \frac{x^n}{n!} & = \left( \sum_{n=0}^{\infty} c_{N,n} \frac{x^n}{n!} \right) \left(1-\sum_{n=0}^{\infty} c_{N-1,n} \frac{x^n}{n!} \right),
\end{align*}
and hence
\begin{align*}
\sum_{n=1}^{\infty} \frac{(1-N)n}{N} \ c_{N-1,n-1} \frac{x^n}{n!}  & =\sum_{n=0}^{\infty} c_{N,n} \frac{x^n}{n!} -
\sum_{m=0}^{\infty} \sum_{k=0}^{\infty} c_{N,m}c_{N-1,k} \frac{x^{m+k}}{m!k!} \\
& =\sum_{n=0}^{\infty} c_{N,n}\ \frac{x^n}{n!} - \sum_{n=0}^{\infty} \sum_{m =0}^{n} c_{N,m} c_{N-1,n-m}\
\frac{x^n}{m!(n-m)!} \\
& =\sum_{n=0}^{\infty} \left\{ c_{N,n}- \sum_{m=0}^{n} \binom{n}{m}c_{N,m}c_{N-1,n-m} \right\} \frac{x^n}{n!}.
\end{align*}
Therefore,  we have 
\begin{align*}
c_{N-1,n-1} & =\frac{N}{(1-N)n}\left\{c_{N,n}-\sum_{m=0}^{n} \binom{n}{m} c_{N,m}c_{N-1,n-m}\right\} \\
& =\frac{N}{(1-N)n}\left\{ -\binom{n}{n-1} c_{N,n-1}c_{N-1,1} -\sum_{m=0}^{n-2}\binom{n}{m} c_{N,m}c_{N-1,n-m}\right\} \\
& =\frac{N}{(1-N)n}\left\{ \frac{(1-N)n}{N}  \ c_{N,n-1}-\sum_{m=0}^{n-2}\binom{n}{m} c_{N,m}c_{N-1,n-m}\right\} \\
& = c_{N,n-1}-\frac{N}{(1-N)n} \sum_{m=0}^{n-2}\binom{n}{m} c_{N,m}c_{N-1,n-m},\\
\end{align*}
for $n\geq 1$, and the proof is complete.
\end{proof}

\begin{proof}[Proof of Proposition \ref{prp4}]
We give the proof by induction for $n$.
In the case $n=1$, the assertion means  $$c_{N,1}=c_{N-1,1}+\frac{N}{1-N} \times \frac{ c_{N-1,0} c_{N-1,2}}{2}, $$ 
and this equality follows from $c_{N,1}=N/(N+1),\ c_{N-1,1}=(N-1)/N, \ c_{N-1,0}=1$ and $c_{N-1,2}=-2(N-1)/N^2(N+1)$.
Assume that the assertion holds up to $n-1$. By Lemma \ref{lem1}, we have

{\small
\begin{align*}
c_{N,n} &=c_{N-1,n} -\frac{N}{(n+1)(N-1)} \sum_{i_1=0}^{n-1}\binom{n+1}{i_1} c_{N,i_1}c_{N-1,n-i_1+1}   \\
& =c_{N-1,n} -\frac{N}{(n+1)(N-1)} \sum_{i_1=0}^{n-1}\binom{n+1}{i_1} c_{N-1,n-i_1+1}   \\
& \quad \times \sum_{m=1}^{i_1} \left( \frac{N}{1-N} \right)^{m-1} \sum_{ 0\leq i_m <\cdots <i_1} \frac{i_1 !}{i_m!}  c_{N-1,i_m}
\prod_{k=2}^m \frac{c_{N-1,i_{k-1}-i_k+1} }{(i_{k-1}-i_k+1)!} \\
& =c_{N-1,n} + \sum_{i_1=0}^{n-1}  \sum_{m=1}^{i_1}\sum_{ 0\leq i_m <\cdots <i_1} 
 \left( \frac{N}{1-N} \right)^{m}  \frac{n !}{i_m!}  c_{N-1,i_m}
\prod_{k=1}^m \frac{c_{N-1,i_{k-1}-i_k+1} }{(i_{k-1}-i_k+1)!} \\
& =c_{N-1,n} +\sum_{m=1}^{n}   \left( \frac{N}{1-N} \right)^{m}  \sum_{ 0\leq i_m <\cdots <i_1<i_0=n} 
\frac{n !}{i_m!}  c_{N-1,i_m}
\prod_{k=1}^m \frac{c_{N-1,i_{k-1}-i_k+1} }{(i_{k-1}-i_k+1)!} \\
& =\sum_{m=0}^{n}   \left( \frac{N}{1-N} \right)^{m}  \sum_{ 0\leq i_m <\cdots <i_1<i_0=n} 
\frac{n !}{i_m!}  c_{N-1,i_m}
\prod_{k=1}^m \frac{c_{N-1,i_{k-1}-i_k+1} }{(i_{k-1}-i_k+1)!}.
\end{align*}
}
\end{proof}
\vspace{1.5cc}

%%%%%%%%%%%%%%%
%%%%%%%%%%%%%%%%
%%%%%%%%%%%%%%%%
\begin{center}
\section{MULTIPLE HYPERGEOMETRIC CAUCHY NUMBERS}
\end{center}

For positive integers $N$ and $r$, define the {\it hypergeometric Cauchy numbers} $c_{N,n}^{(r)}$ by the generating function 
\begin{align}  
\frac{1}{\bigl({}_2 F_1(1,N;N+1;-x)\bigr)^r}&=\left(\frac{(-1)^{N-1}x^N/N}{\log(1+x)-\sum_{n=1}^{N-1}(-1)^{n-1}x^n/n}\right)^r\notag\\
&=\sum_{n=0}^\infty c_{N,n}^{(r)}\frac{x^n}{n!}\,.  
\label{def:higher-hgc} 
\end{align}  

From the definition (\ref{def:higher-hgc}),  
\begin{align*}  
&\left(\frac{(-1)^{N-1}x^N}{N}\right)^r\\
&=\left(\sum_{i=0}^\infty\frac{(-1)^{i+N-1}x^{N+i}}{N+i}\right)^r\left(\sum_{n=0}^\infty c_{N,n}^{(r)}\frac{x^n}{n!}\right)\\
&=x^{r N}(-1)^{r(N-1)}\left(\sum_{l=0}^\infty\sum_{i_1+\cdots+i_r=l\atop i_1,\dots,i_r\ge 0}\frac{(-1)^l l!}{(N+i_1)\cdots(N+i_r)}\frac{x^l}{l!}\right)\left(\sum_{n=0}^\infty c_{N,n}^{(r)}\frac{x^n}{n!}\right)\\
&=x^{r N}(-1)^{r(N-1)}\sum_{n=0}^\infty\sum_{m=0}^n\sum_{i_1+\cdots+i_r=n-m\atop i_1,\dots,i_r\ge 0}\binom{n}{m}\frac{(-1)^{n-m}(n-m)!}{(N+i_1)\cdots(N+i_r)}c_{N,m}^{(r)}\frac{x^n}{n!}\,. 
\end{align*}

Hence, as a generalization of Proposition (\ref{rel:hgc}), for $n\ge 1$, we have the following.  
\\
\begin{Prop}  
$$ 
\sum_{m=0}^n\sum_{i_1+\cdots+i_r=n-m\atop i_1,\dots,i_r\ge 0}\frac{(-1)^{n-m}c_{N,m}^{(r)}}{m!(N+i_1)\cdots(N+i_r)}=0\,.  
$$ 
\label{prp0h} 
\end{Prop}

By using Proposition \ref{prp0h} or 
\begin{equation}  
c_{N,n}^{(r)}=-n! N^r\sum_{m=0}^{n-1}\sum_{i_1+\cdots+i_r=n-m\atop i_1,\dots,i_r\ge 0}\frac{(-1)^{n-m}c_{N,m}^{(r)}}{m!(N+i_1)\cdots(N+i_r)}
\label{eq0h}
\end{equation}  
with $c_{N,0}^{(r)}=1$ ($N\ge 1$), 
some values of $c_{N,n}^{(r)}$ ($0\le n\le 4$) are explicitly given by the following. 
{\small
\begin{align*}  
c_{N,0}^{(r)}&=1\,,\\ 
c_{N,1}^{(r)}&=\frac{r N}{N+1}\,,\\
c_{N,2}^{(r)}&=\frac{r(r+1)N^2}{(N+1)^2}-\frac{2 r N}{N+2}\,,\\ 
c_{N,3}^{(r)}&=\frac{r(r+1)(r+2)N^3}{(N+1)^3}-\frac{6 r(r+1)N^2}{(N+1)(N+2)}
+\frac{6 r N}{N+3}\,,\\  
c_{N,4}^{(r)}&=\frac{r(r+1)(r+2)(r+3)N^4}{(N+1)^4}-\frac{12 r(r+1)(r+2)N^3}{(N+1)^2(N+2)}\\
&\quad +\frac{24 r(r+1)N^2}{(N+1)(N+3)}+\frac{12 r(r+1)N^2}{(N+2)^2}
 -\frac{24 r N}{N+4}\,. 
\end{align*} 
}

As a generalization of Lemma \ref{lem-hgc}, we have an explicit expression of $c_{N,n}^{(r)}$.  

\begin{Prop} 
For $N,n\ge 1$, we have 
$$
c_{N,n}^{(r)}=n!\sum_{k=1}^n(-1)^{n-k}\sum_{e_1+\cdots+e_k=n\atop e_1,\dots,e_k\ge1}D_r(e_1)\cdots D_r(e_k)\,,
$$
where 
\begin{equation} 
D_r(e)=
\sum_{i_1+\cdots+i_r=e\atop i_1,\dots,i_r\ge 0}\frac{N^r}{(N+i_1)\cdots(N+i_r)}\,.  
\label{mre} 
\end{equation}  
\label{prp1h}  
\end{Prop}

The first few values of $D_r(e)$ are given by the following.  
{\small
\begin{align*}  
D_r(1)&=\frac{r N}{N+1}\,,\\ 
D_r(2)&=\frac{r N}{N+2}+\frac{r(r-1)N^2}{2(N+1)^2}\,,\\ 
D_r(3)&=\frac{r N}{N+3}+\frac{r(r-1)N^2}{(N+1)(N+2)}+\binom{r}{3}\frac{N^3}{(N+1)^3}\,,\\ 
D_r(4)&=\frac{r N}{N+4}+\frac{r(r-1)N^2}{(N+1)(N+3)}+\binom{r}{2}\frac{N^2}{(N+1)^2}\\
&\quad +r\binom{r-1}{2}\frac{N^3}{(N+1)^2(N+2)}+\binom{r}{4}\frac{N^4}{(N+1)^4}\,. 
\end{align*} 
}

We shall introduce the Hasse-Teichm\"uller derivative 
in order to prove Proposition \ref{prp1h} easily.    
Let $\mathbb{F}$ be a field of any characteristic, $\mathbb{F}[[z]]$ the ring of formal power series in one variable $z$, and $\mathbb{F}((z))$ the field of Laurent series in $z$. Let $n$ be a nonnegative integer. We define the Hasse-Teichm\"uller derivative $H^{(n)}$ of order $n$ by 
$$
H^{(n)}\left(\sum_{m=R}^{\infty} c_m z^m\right)
=\sum_{m=R}^{\infty} c_m \binom{m}{n}z^{m-n}
$$
for $\sum_{m=R}^{\infty} c_m z^m\in \mathbb{F}((z))$, 
where $R$ is an integer and $c_m\in\mathbb{F}$ for any $m\geq R$. Note that $\binom{m}{n}=0$ if $m<n$.  

The Hasse-Teichm\"uller derivatives satisfy the product rule \cite{Teich}, the quotient rule \cite{GN} and the chain rule \cite{Hasse}. 
One of the product rules can be described as follows.  
\\
\begin{Lem}  
For $f_i\in\mathbb F[[z]]$ $(i=1,\dots,k)$ with $k\ge 2$ and for $n\ge 1$, we have 
$$
H^{(n)}(f_1\cdots f_k)=\sum_{i_1+\cdots+i_k=n\atop i_1,\dots,i_k\ge 0}H^{(i_1)}(f_1)\cdots H^{(i_k)}(f_k)\,. 
$$ 
\label{productrule2}
\end{Lem} 

The quotient rules can be described as follows.  \\
\begin{Lem}  
For $f\in\mathbb F[[z]]\backslash \{0\}$ and $n\ge 1$,  
we have 
\begin{align} 
H^{(n)}\left(\frac{1}{f}\right)&=\sum_{k=1}^n\frac{(-1)^k}{f^{k+1}}\sum_{i_1+\cdots+i_k=n\atop i_1,\dots,i_k\ge 1}H^{(i_1)}(f)\cdots H^{(i_k)}(f)
\label{quotientrule1}
\\ 
&=\sum_{k=1}^n\binom{n+1}{k+1}\frac{(-1)^k}{f^{k+1}}\sum_{i_1+\cdots+i_k=n\atop i_1,\dots,i_k\ge 0}H^{(i_1)}(f)\cdots H^{(i_k)}(f)\,.
\label{quotientrule2} 
\end{align}   
\label{quotientrules}
\end{Lem}

\begin{proof}[Proof of Proposition \ref{prp1h}]  
Put $h(x)=\bigl(f(x)\bigr)^r$, where 
$$
f(x)=\dfrac{\sum_{i=N}^\infty(-1)^{i-1}\frac{x^i}{i}}{(-1)^{N-1}\frac{x^N}{N}}=\sum_{j=0}^\infty\frac{(-1)^j N}{N+j}x^j\,. 
$$ 
Since 
\begin{align*} 
\left.H^{(i)}(f)\right|_{x=0}&=\left.\sum_{j=i}^\infty\frac{(-1)^j N}{N+j}\binom{j}{i}x^{j-i}\right|_{x=0}\\
&=
\frac{(-1)^i N}{N+i}
\end{align*}  
by the product rule of the Hasse-Teichm\"uller derivative in Lemma \ref{productrule2}, we get 
\begin{align*} 
\left.H^{(e)}(h)\right|_{x=0}&=\sum_{i_1+\cdots+i_r=e\atop i_1,\dots,i_r\ge 0}\left.H^{(i_1)}(f)\right|_{x=0}\cdots\left.H^{(i_r)}(f)\right|_{x=0}\\
&=\sum_{i_1+\cdots+i_r=e\atop i_1,\dots,i_r\ge 0}\frac{(-1)^{i_1}N}{N+i_1}\cdots\frac{(-1)^{i_r}N}{N+i_r}\\
&=(-1)^e\sum_{i_1+\cdots+i_r=e\atop i_1,\dots,i_r\ge 0}\frac{N^r}{(N+i_1)\cdots(N+i_r)}:=(-1)^e D_r(e)\,. 
\end{align*} 
Hence, by the quotient rule of the Hasse-Teichm\"uller derivative in Lemma \ref{quotientrules} (\ref{quotientrule1}), we have 
\begin{align*} 
\frac{c_{N,n}^{(r)}}{n!}&=\sum_{k=1}^n\left.\frac{(-1)^k}{h^{k+1}}\right|_{x=0}\sum_{e_1+\cdots+e_k=n\atop e_1,\dots,e_k\ge 1}\left.H^{(e_1)}(h)\right|_{x=0}\cdots\left.H^{(e_k)}(h)\right|_{x=0}\\
&=\sum_{k=1}^n(-1)^k\sum_{e_1+\cdots+e_k=n\atop e_1,\dots,e_k\ge1}(-1)^n D_r(e_1)\cdots D_r(e_k)\,. 
\end{align*} 
\end{proof}

Now, we can also show a determinant expression of $c_{N,n}^{(r)}$.   

\begin{theorem}  
For $N,n\ge 1$, we have 
$$
c_{N,n}^{(r)}=n!\left|
\begin{array}{ccccc} 
D_r(1)&1&&&\\  
D_r(2)&D_r(1)&&&\\ 
\vdots&\vdots&\ddots&1&\\ 
D_r(n-1)&D_r(n-2)&\cdots&D_r(1)&1\\ 
D_r(n)&D_r(n-1)&\cdots&D_r(2)&D_r(1) 
\end{array} 
\right|\,. 
$$ 
where $D_r(e)$ are given in $(\ref{mre})$. 
\label{th1-r}  
\end{theorem}  

\noindent 
{\it Remark.}  
When $r=1$ in Theorem \ref{th1-r}, we have the result in Theorem \ref{th-cau1}.  

\begin{proof} 
For simplicity, put $b_{N,n}^{(r)}=c_{N,n}^{(r)}/n!$. Then, we shall prove that for any $n\ge 1$ 
\begin{equation}  
b_{N,n}^{(r)}=\left|
\begin{array}{ccccc} 
D_r(1)&1&&&\\  
D_r(2)&D_r(1)&&&\\ 
\vdots&\vdots&\ddots&1&\\ 
D_r(n-1)&D_r(n-2)&\cdots&D_r(1)&1\\ 
D_r(n)&D_r(n-1)&\cdots&D_r(2)&D_r(1) 
\end{array} 
\right|\,.
\label{aNnr}
\end{equation}   
When $n=1$, (\ref{aNnr}) is valid because 
$$
D_r(1)=\frac{r N^r}{N^{r-1}(N+1)}=\frac{r N}{N+1}=b_{N,1}^{(r)}\,. 
$$ 
Assume that (\ref{aNnr}) is valid up to $n-1$. Notice that 
by (\ref{eq0h}), we have 
$$
b_{N,n}^{(r)}=\sum_{l=1}^{n}(-1)^{l-1}b_{N,n-l}^{(r)}D_r(l)\,. 
$$ 
Thus, by expanding the first row of the right-hand side (\ref{aNnr}), it is equal to 
\begin{align*} 
&D_r(1)b_{N,n-1}^{(r)}-\left|
\begin{array}{ccccc} 
D_r(2)&1&&&\\  
D_r(3)&D_r(1)&&&\\ 
\vdots&\vdots&\ddots&1&\\ 
D_r(n-1)&D_r(n-3)&\cdots&D_r(1)&1\\ 
D_r(n)&D_r(n-2)&\cdots&D_r(2)&D_r(1) 
\end{array} 
\right|\\
&=D_r(1)b_{N,n-1}^{(r)}-D_r(2)b_{N,n-2}^{(r)}\\
&\qquad +\left|
\begin{array}{ccccc} 
D_r(3)&1&&&\\  
D_r(4)&D_r(1)&&&\\ 
\vdots&\vdots&\ddots&1&\\ 
D_r(n-1)&D_r(n-4)&\cdots&D_r(1)&1\\ 
D_r(n)&D_r(n-3)&\cdots&D_r(2)&D_r(1) 
\end{array} 
\right|\\
&=D_r(1)b_{N,n-1}^{(r)}-D_r(2)b_{N,n-2}^{(r)}+\cdots
+(-1)^{n-2}\left|
\begin{array}{cc}
D_r(n-1)&1\\
D_r(n)&D_r(1)
\end{array} 
\right|\\
&=\sum_{l=1}^n(-1)^{l-1}D_r(l)b_{N,n-l}^{(r)}=b_{N,n}^{(r)}\,.
\end{align*} 
Note that $b_{N,1}^{(r)}=D_r(1)$ and $b_{N,0}^{(r)}=1$. 
\end{proof} 
\begin{center}
\section{A RELATION BETWEEN $c_{N,n}^{(r)}$ AND  $c_{N,n}$}
\end{center}
In this section, we show the following relation between $c_{N,n}^{(r)}$ and $c_{N,n}$.
\\
\begin{Lem}
For $r \geq 2$ and $N, n\geq 0$, we have
$$
c_{N,n}^{(r)}=
\sum_{n_1,\ldots,n_r\geq 0 \atop 
n_1+\cdots +n_r=n} \frac{n!}{n_1!\cdots n_r!} \ C_{N,n_1}\cdots C_{N,n_r},
$$
\end{Lem}
\begin{proof}
From the definition (\ref{def:higher-hgc}), we have
\begin{align*}
\sum_{n=0}^{\infty} C_{N,n}^{(r)} \frac{n!}{x^n} &=\left( \sum_{n=0}^{\infty}
C_{N,n}\frac{x^n}{n!} \right)^r\\
&=\sum_{n=0}^{\infty} \sum_{n_1,\ldots,n_r\geq 0 \atop 
n_1+\cdots +n_r=n} \frac{n!}{n_1!\cdots n_r!}\  C_{N,n_1}\cdots C_{N,n_r} \frac{x^n}{n!},
\end{align*}
and we get the assertion.
\end{proof}
\bigskip
\noindent
{\it Examples}
\begin{itemize}
\item[{\rm (i)}] $c_{N,0}^{(r)}=c_{N,0}^{r+1}$
\item[{\rm (ii)}] $c_{N,1}^{(r)}=rc_{N,1}$
\item[{\rm (iii)}] $c_{N,2}^{(r)}=r c_{N,2}c_{N,0}^{N-1}+r(r-1)c_{N,1}^2c_{N,0}^{r-2}$
\end{itemize}
%\end{proof}

%%%%%%%%%%%%%%%%%
\begin{center}
\section{APPLICATIONS BY THE TRUDI'S FORMULA AND INVERSION EXPRESSIONS}
\end{center}

The expressions in Theorem \ref{th-cau1} and Theorem \ref{th1-r} are useful for several applications too.   
We can obtain different explicit expressions for the numbers $c_{N,n}^{(r)}$, $c_{N,n}$ and $c_{n}$ by using the Trudi's formula. We also show some inversion formulas.  The following relation is known as Trudi's formula \cite[Vol.3, p.214]{Muir},\cite{Trudi} and the case $a_0=1$ of this formula is known as Brioschi's formula \cite{Brioschi},\cite[Vol.3, pp.208--209]{Muir}.  
\\
\begin{Lem}
For a positive integer $m$, we have 
\begin{multline*} 
\left|
\begin{array}{ccccc}
a_1  & a_2   &  \cdots   & \cdots& a_m  \\
a_{0}  & a_{1}    & \ddots   & &  \vdots \\
 &  \ddots&  \ddots  &  \ddots & \vdots  \\
  &     & \ddots  &a_1  & a_{2}  \\
&    &    & a_0  & a_1
\end{array}
\right|\\
=
\sum_{t_1 + 2t_2 + \cdots +mt_m=m}\binom{t_1+\cdots + t_m}{t_1, \dots,t_m}(-a_0)^{m-t_1-\cdots - t_m}a_1^{t_1}a_2^{t_2}\cdots a_m^{t_m}, \label{trudi}
\end{multline*}
where $\binom{t_1+\cdots + t_m}{t_1, \dots,t_m}=\frac{(t_1+\cdots + t_m)!}{t_1 !\cdots t_m !}$ are the multinomial coefficients. 
\label{lema0}
\end{Lem}

In addition, there exists the following inversion formula, which is based upon the relation: 
$$
\sum_{k=0}^n(-1)^{n-k}\alpha_k R(n-k)=0\quad(n\ge 1)\,. 
$$ 

\begin{Lem}
If $\{\alpha_n\}_{n\geq 0}$ is a sequence defined by $\alpha_0=1$ and 
$$
\alpha_n=\begin{vmatrix} R(1) & 1 & & \\
R(2) & \ddots &  \ddots & \\
\vdots & \ddots &  \ddots & 1\\
R(n) & \cdots &  R(2) & R(1) \\
 \end{vmatrix},  \ \text{then} \ R(n)=\begin{vmatrix} \alpha_1 & 1 & & \\
\alpha_2 & \ddots &  \ddots & \\
\vdots & \ddots &  \ddots & 1\\
\alpha_n & \cdots &  \alpha_2 & \alpha_1 \\
 \end{vmatrix}\,.
$$
Moreover, if 
$$
A=\begin{pmatrix} 
  1 &  & & \\
\alpha_1 & 1  &   & \\
\vdots & \ddots &  \ddots & \\
\alpha_n& \cdots &  \alpha_1 & 1 \\
 \end{pmatrix}, \ \text{then} \  A^{-1}=\begin{pmatrix} 
  1 &  & & \\
R(1) & 1  &   & \\
\vdots & \ddots &  \ddots & \\
R(n) & \cdots &  R(1) & 1 \\
 \end{pmatrix}\,.
$$
\label{lema}
\end{Lem}

From Trudi's formula, it is possible to give the combinatorial expression  
$$
\alpha_n=\sum_{t_1+2t_2+\cdots +n t_n=n}\binom{t_1+\cdots+t_n}{t_1, \dots, t_n}(-1)^{n-t_1-\cdots - t_n}R(1)^{t_1}R(2)^{t_2}\cdots R(n)^{t_n}\,.
$$
By applying these lemmata to Theorem \ref{th1-r}, we obtain an explicit expression for the generalized hypergeometric Cauchy  numbers $c_{N,n}^{(r)}$.  A different version can be seen in \cite{KY}.   

\begin{theorem}\label{Turdi1}
For $n\geq 1$ 
\begin{multline*} 
c_{N,n}^{(r)}\\
=n!\sum_{t_1 + 2t_2 + \cdots + nt_n=n}\binom{t_1+\cdots + t_n}{t_1, \dots,t_n}(-1)^{n-t_1-\cdots-t_n}D_r(1)^{t_1}D_r(2)^{t_2}\cdots D_r(n)^{t_n}\,, 
\end{multline*}  
where $D_r(e)$ are given in $(\ref{mre})$. 
Moreover, 
$$
D_r(n)=\begin{vmatrix} \frac{c_{N,1}^{(r)}}{1!} & 1 & & \\
\frac{c_{N,2}^{(r)}}{2!} & \ddots &  \ddots & \\
\vdots & \ddots &  \ddots & 1\\
\frac{c_{N,n}^{(r)}}{n!} & \cdots &  \frac{c_{N,2}^{(r)}}{2!} & \frac{c_{N,1}^{(r)}}{1!} \\
 \end{vmatrix}\,,
$$
and  
{\small 
$$  
\begin{pmatrix} 1 &  &  & & \\
\frac{c_{N,1}^{(r)}}{1!} & 1 &   &  & \\
\frac{c_{N,2}^{(r)}}{2!} & \frac{c_{N,1}^{(r)}}{1!}  &  1  &  & \\
\vdots &  &  \ddots &  \ddots & \\
\frac{c_{N,n}^{(r)}}{n!} & \cdots &  \frac{c_{N,2}^{(r)}}{2!} & \frac{c_{N,1}^{(r)}}{1!} &1 
\end{pmatrix}^{-1} 
=\begin{pmatrix} 1 &  &  & & \\
D_r(1) & 1 &   &  & \\
D_r(2) & D_r(1)  &  1  &  & \\
\vdots &  &  \ddots &  \ddots & \\
D_r(n) & \cdots &  D_r(2) & D_r(1) &1 
\end{pmatrix}\,.
$$  
}  
\label{th1234}
\end{theorem}

When $r=1$ in Theorem \ref{th1234}, we have  an explicit expression for the numbers $c_{N,n}$.  

\begin{Cor}  
For $n\geq 1$ 
\begin{multline*} 
c_{N,n} 
=n!\sum_{t_1 + 2t_2 + \cdots + nt_n=n}\binom{n-t_1-\cdots - t_n}{t_1, \dots,t_n}(-1)^{t_1+\cdots+t_n}\\
\times\left(\frac{N}{N+1}\right)^{t_1}\left(\frac{N}{N+2}\right)^{t_2}\cdots \left(\frac{N}{N+n}\right)^{t_n}  
\end{multline*}  
and  
$$
\frac{N}{N+n}=\begin{vmatrix} \frac{c_{N,1}}{1!} & 1 & & \\
\frac{c_{N,2}}{2!} & \ddots &  \ddots & \\
\vdots & \ddots &  \ddots & 1\\
\frac{c_{N,n}}{n!} & \cdots &  \frac{c_{N,2}}{2!} & \frac{c_{N,1}}{1!} \\
 \end{vmatrix}\,,
$$
\label{cor1}
\end{Cor}

When $r=N=1$ in Theorem \ref{th1234}, we have a different expression of the classical Cauchy numbers. 

\begin{Cor}  
We have for $n\ge 1$ 
\begin{multline*} 
c_{n} 
=n!\sum_{t_1 + 2t_2 + \cdots + nt_n=n}\binom{t_1+\cdots + t_n}{t_1, \dots,t_n}(-1)^{n-t_1-\cdots-t_n}\\
\times\left(\frac{1}{2}\right)^{t_1}\left(\frac{1}{3}\right)^{t_2}\cdots \left(\frac{1}{n+1}\right)^{t_n}  
\end{multline*}  
and  
$$
\frac{1}{n+1}=\begin{vmatrix} \frac{c_{1}}{1!} & 1 & & \\
\frac{c_{2}}{2!} & \ddots &  \ddots & \\
\vdots & \ddots &  \ddots & 1\\
\frac{c_{n}}{n!} & \cdots &  \frac{c_{2}}{2!} & \frac{c_{1}}{1!} \\
 \end{vmatrix}.
$$
\label{cor2}
\end{Cor}

%%%%%%%%%%%%%%%%%%
%%%%%%%%%%%%%%%%%%%
%%%%%%%%%%%%%%%%%%%%
%\begin{color}{red}  
\section{ADDITIONAL COMMENTS}  

Hypergeometric Cauchy numbers are not integers, but fractions.  Hence, combinatorial interpretations of the above results or congruent relations seem to be difficult to obtain.  
Nevertheless, definition (\ref{def:hgc}) is not obvious or artificial, but has motivations from Combinatorics, in particular, graph theory.  In 1989, Cameron \cite{Cameron} considered the operator $A$ defined on the set of sequences of non-negative integers as follows. For $x=\{x_n\}_{n\ge 1}$ and $z=\{z_n\}_{n\ge 1}$, let $A x=z$, where 
\begin{equation}  
1+\sum_{n=1}^\infty z_n t^n=\left(1-\sum_{n=1}^\infty x_n t^n\right)^{-1}\,.
\label{cameron}
\end{equation}    
For hypergeometric Cauchy numbers, we have 
$$
x_n=\frac{c_{N,n}}{n!}\quad\hbox{and}\quad z_n=\frac{(-1)^{n-1}N}{N+n}
$$ 
and vice versa.  
If $x$ enumerates a class $C$, then $A x$ enumerates the class of disjoint unions of members. More concrete examples can be seen in \cite{Cameron}. 
%\end{color} 
%%%%%%%%%%%%%%%%%%%%%
%%%%%%%%%%%%%%%%%%%
%%%%%%%%%%%%%%%%%%%

%%%%%%%%%%%%%%%%%%
\begin{center}

\end{center}
\end{document}